\documentclass{alice.v2}
\newtheorem{thm}{Theorem}
\usepackage{amsmath,amssymb,graphicx,latexsym,mathabx,float}
\usepackage{hyperref}
\usepackage[font=scriptsize]{caption}

\newcommand{\nin}{\not\in}
\newtheorem{corollary}[thm]{Corollary}
\newtheorem{lemma}{Lemma}
\newtheorem{theorem}[thm]{Theorem}
\newtheorem{proposition}[thm]{Proposition}
\newtheorem*{lemma*}{Lemma}
\theoremstyle{definition}
\newtheorem{remark}[thm]{Remark}
\newtheorem{definition}[thm]{Definition}
\newcommand{\seco}{\searrow\!\!\!\searrow }     

\title[A strong collapse increasing geometric simplicial
LS category]{A strong collapse increasing the geometric simplicial
		Lusternik-Schnirelmann category}	
	\author{Dimitris Askitis}
	\address{Department of Mathematical Sciences\\
	University of Copenhagen\\
	Universitetsparken 5\\
	Copenhagen 2100\\
	Denmark}
\email{dimitrios@math.ku.dk}
\begin{document}

\begin{abstract}
	In {\cite{scat1}}, after defining notions of LS category in the simplicial context, the authors show that the geometric simplicial LS category is non-decreasing under strong collapses. However, they do not give examples where it increases strictly, but they conjecture that such an example should exist, and thus that the
	geometric simplicial LS category is not strong homotopy invariant. The
	purpose of this note is to provide with such an example. We construct a simplicial
	complex whose simplicial and geometric simplicial LS categories are different, and using
	this, we provide an example of a strong collapse that increases the
	geometric simplicial LS category, thus settling the geometric simplicial LS category
	not being strong homotopy invariant.
\end{abstract}
\maketitle

\section{Introduction}

The Lusternik-Schnirelmann category (for short, LS category) for a
topological space $X$ is defined as the smallest integer $k$ such that there
is an open covering $\{ U_j)_{j \leqslant k + 1}$ of cardinality $k + 1$ of
$X$ such that the inclusion maps $i_j : U_j \hookrightarrow X$ are
nullhomotopic. It is an important homotopy invariant, providing an upper bound
for the critical points of a manifold, among several other applications..

In {\cite{scat1}}, the authors introduced the simplicial LS category, i.e. a
notion of LS category for simplicial complexes. The advantage of their
simplicial version is that it is a strongly homotopy invariant and it depends
only in the simplicial structure, not on the chosen geometric realisation. For
further information on the simplicial LS category, see also {\cite{scat2}}. A
nice introduction, in relation with finite topological spaces and using
category theoretic language, may also be found in {\cite{erica}}.

Let $K, L$ be simplicial complexes. Two simplicial maps $\varphi, \psi : K
\rightarrow L$ are said to be \textit{contiguous} if for every simplex
$\sigma \in K$, $\varphi (\sigma) \cap \psi (\sigma)$ is a simplex in $L$. The
contiguity relation is denoted by $\varphi \sim_c \psi$. The relation $\sim_c$
is symmetric and reflexive, but in general it is not transitive. Thus, as a
simplicial equivalent to homotopy, the notion of the contiguity class needs to
be introduced:

\begin{definition}
  Let $K, L$ be simplicial complexes. Two simplicial maps $\varphi, \psi : K
  \rightarrow L$ belong to the same contiguity class ($\varphi \sim \psi$) if
  there is a sequence $\{ \varphi_i \}_{i \leqslant n}$ of maps from $K$ to
  $L$ such that $\varphi = \varphi_0 \sim_c \varphi_1 \sim_c \varphi_2 \sim_c
  \cdots \sim_c \varphi_n = \psi$.
\end{definition}
The role of sets whose inclusion is nullhomotopic is to be played by
categorical subcomplexes.
\begin{definition}
  Let $K$ be a simplicial complex. We say that a subcomplex $U \subset K$ is
  categorical if there exists a vertex $v \in K$ such that the inclusion map
  $i_U : U \hookrightarrow K$ and the constant map $c_v$ are in the same
  contiguity class, i.e. $i_U \sim c_v$.
\end{definition}

\begin{definition}
  The simplicial LS category of a simplicial complex $K$, denoted by
  $\text{scat} K$, is the least integer $k$ such that $K$ can be covered by $k
  + 1$ categorical subcomplexes. Such a cover is called categorical.
\end{definition}

Let $K$ be a simplicial complex and $u, v \in K$ be two vertices. If for every
maximal simplex $\tau \in K$ such that $u \in \tau$, we have that $v \in K$,
we say that $u$ is \textit{dominated} by $v$. Deleting such a vertex and
removig all simplices that contain it is called \textit{elementary strong
collapse} and results in the simplicial complex $K \setminus u$. We say that
$K$ \textit{strong collapses} to $L$ if there is a sequence of elementary
strong collapses from $K$ to $L$, and we denote it by $K \seco L$.
The inverse procedure, going from $L$ back to $K$ by adding dominated
vertices, is called a \textit{strong expansion}. We say that $K$ and $L$
have the same strong homotopy type if there is a sequence of strong collapses
and expansions from $K$ to $L$. A well known result (see {\cite[Corollary
2.12]{Barmak2012}}) is that $K, L$ have the same strong homotopy type if and
only if there are maps $\varphi : K \rightarrow L$ and $\psi : L \rightarrow
K$ such that $\varphi \circ \psi \sim \text{Id}_L$ and $\psi \circ \varphi
\sim \text{Id}_K$. Then, we denote $K \sim L$. The simplicial LS category is a
strong homotopy invariant.

Another related notion is that of the geometric simplicial LS category, based
on the notion of strong collapsibility. As conjectured in \cite{scat1} and
proven in the present note, this is not a strong homotopy invatiant. There are
similar notions based on simple collapsibility, see {\cite{aaronson}}.

\begin{definition}
  Let $K$ be a simplicial complex. We say that $K$ is strongly collapsible if
  it strongly collapses to a point, i.e. if $\text{Id}_K \sim c_v$, where
  $c_v$ is the constant map to a vertex $v \in K$.
\end{definition}

\begin{definition}
  The geometric simplicial category of a simplicial complex $K$, denoted by
  $\text{gscat} K$, is the least integer $k$ such that $K$ can be covered by
  $k + 1$ strongly collapsible subcomplexes. Such a cover is called geometric.
\end{definition}

By the definitions above, a strongly collapsible subcomplex is also
categorical, but the opposite does not always hold. In fact, a categorical
subcomplex need not even be connected, while a strongly collapsible one is
necessarily connected. The following results relate these two different
notions of LS category for simplicial complexes.

\begin{proposition}
  {\cite[Proposition 4.2]{scat1}} Let $K$ be a simplicial complex. Then,
  $\textnormal{scat} K \leqslant \textnormal{gscat} K$.\label{prop1}
\end{proposition}

\begin{theorem}
  {\cite[Theorem 4.3]{scat1}} Let $M, K$ be simplicial complexes such that $K$
  is a strong collapse of $M$. Then, $\textnormal{gscat} M \leqslant \textnormal{gscat}
  K$.\label{prop2}
\end{theorem}

The authors in \cite{scat1} remark that they did not have an example that
satisfies the strict inequality in Theorem \ref{prop2}, but that, based on
similar results regarding partially ordered sets and beat points, such an
example should exist. In this short note, our purpose is to construct examples
that satisfy these inequalities strictly, i.e. showing that the inequalities
above do not degenerate to equalities. First, in section $2$, we construct an
example that satisfies the strict inequality in Proposition \ref{prop1}, with
the further property that it has minimal categorical cover which consists of
connected subcomplexes. This example is important in constructing a second one
satisfying the strict inequality in Theorem \ref{prop2}, as the simplicial LS
category is strong homotopy invariant, hinting that this second example should
be searched among cases where simplicial and geometric simplicial LS categories
are not equal.

The main contribution of this note is the next proposition, proven in
sections $2$ and $3$, and the follow-up corollary, derived directly from it.

\begin{proposition}
  There is a simplicial complex $K$ such that $\textnormal{scat} K < \textnormal{gscat}
  K$. Moreover, there is a simplicial complex $M$ such that $M \seco
  K$ and $\textnormal{gscat} M < \textnormal{gscat} K$\label{myprop}.
\end{proposition}

\begin{corollary}
  The geometric simplicial LS category is not strong homotopy invariant.
\end{corollary}

\begin{remark}
  The respective geometric realisations provide examples for the topological
  analogues of the strict inqualities. Of course, examples to these have
  already been known.
\end{remark}

\section{Example of strictly bigger geometric simplicial than simplicial LS category}

In this section, we shall construct a simplicial complex $K$ such that
$\text{gscat} K > \text{scat} K$. Let $K$ be the simplicial complex with set
of vertices
\[ V_K = \{ (k, l) | - 2 \leqslant k \leqslant 2, 0 \leqslant l \leqslant 2 \}
\]
and simplicial structure as shown in the figure below, where the rightmost and
leftmost vertices, the rightmost and leftmost edges, and the vertices with
coordinates $(0, 0)$ (noted by the bulk points) are identified:
\begin{figure}[H]\centering
	{\includegraphics[height=6cm,keepaspectratio]{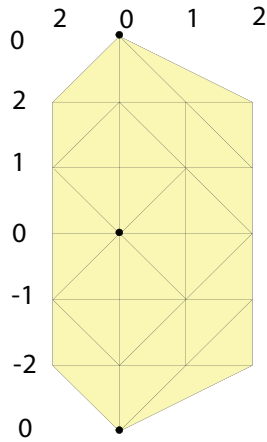}}
  \caption{The simplicial complex $K$. The rightmost and leftmost edges and\label{fig1}
  vertices, as well as the $3$ bulk points corresponding to the vertex $(0,
  0)$, are identified.}
\end{figure}
It has degree 2 and 15 vertices, $45$ edges and 30 $2$-simplices.
\begin{figure}[H]\centering
	\includegraphics[height=5cm,keepaspectratio]{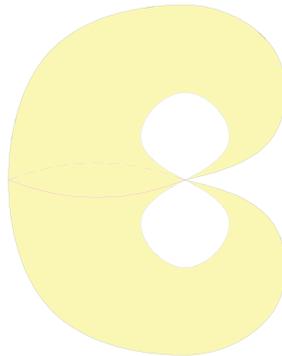}
  \caption{A geometric realisation of $K$. It is a sphere in which we have
  created $2$ handles by identifying $3$ points together.}
\end{figure}
\newpage
It is clear that $K$ is not strongly collapsible (its geometric realisation
is not even contractible), hence $\text{scat} K > 0$. The following
categorical cover has cardinality $2$, hence $\text{scat} K = 1$.
\begin{figure}[H]\centering
	\includegraphics[height=6cm,keepaspectratio]{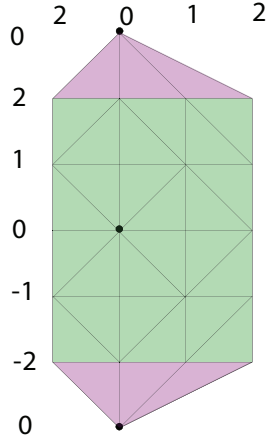}
  \caption{A categorical cover of $K$}
\end{figure}
However, this cover is not geometric. The green subcomplex is not strongly
collapsible, as its core is not trivial. It can be strongly collapsed into the
loop $\{ \{ (0, 2), (0, 0) \}, \{ (0, 0), (0, 1) \}, \{ (0, 2), (0, 1) \} \}$
that cannot be further strongly collapsed as it contains no dominated points.

The green subcomplex is categorical, but no strongly collapsible subcomplex
contains it. This is an important fact as, otherwise, just by extending it we
would have gotten a geometric cover, and thus the simplicial and the geometric
simplicial LS categories would be equal. Hence, the existence of a categorical
subcomplex which cannot be contained in a bigger strongly collapsible
subcomplex is necessary for the simplicial and geometric simplicial LS
categories to be different.

It is not difficult to see through combinatorial arguments that it is not
possible to cover $K$ with $2$ strongly collapsible subcomplexes.
Heuristically, assume that $K$ is covered by strongly collapsible subcomplexes $A$ and $B$, and wlog that one of the triangles on the top of the figure is solely
contained in $A$. For the moment, let's treat the vertices $(0,0)$ in the $3$ different places in Figure \ref{fig1} as not being identified. As no sequence of edges in $A$ may connect the top $(0,0)$ with the middle or the bottom ones (as these loops cannot be in any strongly collapsible subcomplex in $K$), there must be a maximal connected subcomplex in $A$ that contains this triangle, and no triangles in the middle or bottom that contain $(0,0)$. The boundary of this must then reside in $A\cap B$, and it must contain a loop that winds around one of the $(0,0)$'s. But, for such a loop, there can only be a single minimal strongly collapsible subcomplex containing it. This subcomplex cannot be in both $A$ and $B$ as it must either contain the middle or bottom $(0,0)$, thus it cannot be contained in $A$, or the top $(0,0)$, thus it cannot contain $B$ by our initial assumptions (that there is a triangle in the top not contained in $B$). Thus $K$ cannot be covered by $2$ strongly collapsible subcomplexes. For a more detailed proof, see
Appendix. Hence, $\text{gscat} K > 1$.

The following figure shows a geometric cover of cardinality $3$, which implies
that $\text{gscat} K = 2$. Thus, the first part of Proposition \ref{myprop} is
proven.

\begin{figure}[H]\centering
	\includegraphics[height=6cm,keepaspectratio]{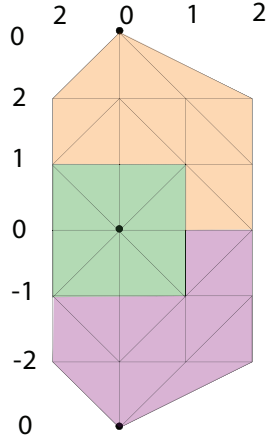}
  \caption{A geometric cover of $K$ consisting of $3$ strongly collapsible
  subcomplexes.}
\end{figure}

\section{A strong collapse that increases the geometric simplicial LS
category}

Based on the previous example, we consider $K$ as before. Let $M$ be a
simplicial complex with vertices
\[ V_M = V_K \cup \{ a \} \]
and
\begin{align*}
M = K &\cup \{ \{ a \} \} \cup \{ \{ a, (0, 0) \}, \{ a, (2, 0) \}, \{ a,
(2, 1) \}, \{ a, (2, 2) \} \}\\ &\cup \left\{ \{ a, (0, 0), (2, 0) \}, \{ a,
(0, 0), (2, 1) \}, \{ a, (0, 0), (2, 2) \}\right\}\\&\cup\left\{ \{ a, (2, 0), (2, 1) \}, \{ a,
(2, 1), (2, 2) \}, \{ a, (2, 2), (2, 0) \} \right\} \\&\cup \{ \{ a, (0, 0),
(2, 0), (2, 1) \}, \{ a, (0, 0), (2, 1), (2, 2) \}, \{ a, (0, 0), (2, 2),
(2, 0) \} \}
\end{align*}
What we have essentially constructed is a simplicial complex in which we have
added a new vertex $a$ and filled in the $3$-simplices that contain $a$, $(0,
0)$ and two of the $(2, 0), (2, 1), (2, 2)$. We see that the new vertex $a$ is
dominated by $(0, 0)$, so $M \seco K$.

\begin{figure}[H]
  \raisebox{0\height}{\includegraphics[height=6cm,keepaspectratio]{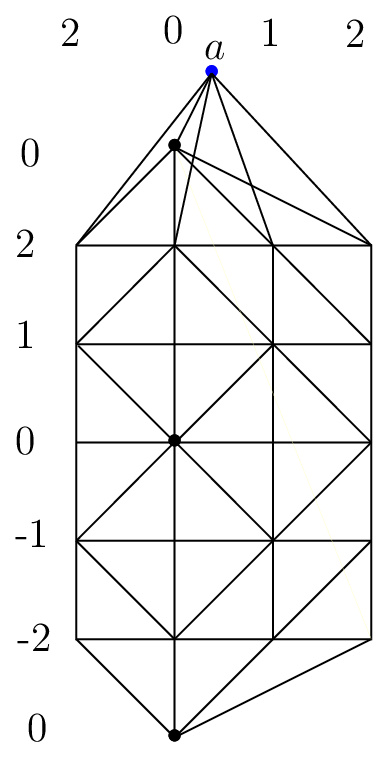}}
  \caption{The $1$-skeleton of $M$. The bulk black points and the rightmost
  and leftmost vertices/edges are identified.}
\end{figure}

Let $A \subset M$ be the subcomplex whose set of maximal simplices is
\begin{align*}
&\{ \{ a, (0, 0), (2, 0), (2, 1) \}, \{ a, (0, 0), (2, 1), (2, 2) \}, \{ a,
(0, 0), (2, 2), (2, 0) \} \}\cup\\  &\{ \{ (0, 0), (- 2, 0), (- 2, 1) \}, \{
(0, 0), (- 2, 1), (- 2, 2) \}, \{ (0, 0), (- 2, 2), (- 2, 0) \} \} 
\end{align*}
We can easily see that $A$ strongly collapses to $(0, 0)$, by first deleting
the dominated vertex $a$, and then the other vertices. Let $B_0$ be the
subcomplex consisting of all the maximal simplices of $M$ that are not in $A$,
$B_1$ the one consisting of the $2$-simplices $\{ \{ a, (2, 0), (2, 1) \}, \{
a, (2, 1), (2, 2) \}, \{ a, (2, 2), (2, 0) \} \}$, and $B = B_0 \cup B_1$.
Again, one can show that $B$ strongly collapses to $a$: $B_0$ strongly
collapses to the simplicial complex consisting of three $1$-simplices $\{ \{
(2, 0), (2, 1) \}, \{ (2, 1), (2, 2) \}, \{ (2, 2), (2, 0) \} \}$, and then
$(2, 0), (2, 1)$ and $(0, 1)$ are dominated by $a$, hence $B_0 \cup B_1$
strongly collapses to $a$.

Thus, we have constructed a cover of $M$ consisting of $2$ strongly
collapsible subcomplexes, and hence $\text{gscat} M = 1$, while we have seen
that $\text{gscat} K$=2 and $M \seco K$, showing the second part of
Proposition \ref{myprop}.

\appendix\section*{Appendix}

\begin{lemma}
  The simplicial complex $K$, as in section $2$, cannot be covered by two
  strongly collapsible subcomplexes.
\end{lemma}

\begin{proof}
  We will argue by contradiction. The idea is that a strongly collapsible
  subcomplex may not contain certain loops, e.g. a loop starting from $(0, 0)$
  at the top of the figure and ending at $(0, 0)$ in the bottom or the middle.
  Assume that there exist two strongly collapsible subcomplexes $A, B \subset
  K$ such that $A \cup B = K$. Out of the three edges $\{ (0, 0), (2, 0) \}$,
  $\{ (2, 0), (1, 0) \}$ and $\{ (1, 0), (0, 0) \}$ one has to belong solely
  to $A$ and one solely to $B$, as no strongly collapsible subcomplex may
  contain the loop $\left\{\{ (0, 0), (2, 0) \}, \{ (2, 0), (1, 0) \},\{
  (1, 0), (0, 0) \}\right\}$. Let's enumerate the triangles in the following
  way and assume that $\{ (0, 0), (2, 0) \} \nin B$ and $\{ (1, 0), (0, 0) \}
  \nin A$. The other cases can be argued in a similar way.
  
  \begin{figure}[H]\centering
  	\includegraphics[height=6cm,keepaspectratio]{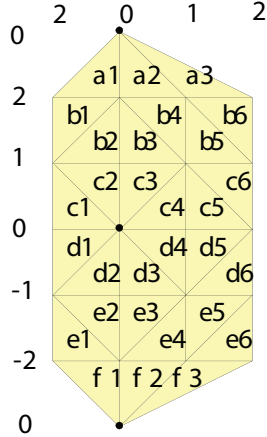}
    \caption{An enumeration of the $2$-simplices of $K$.}
  \end{figure}
  
  We must have that $a 1, a 2 \in A$ and $c 2, c 3 \in B$. Then, $\{ (1, 1),
  (1, 2) \} \nin B$, as if the loop $\{ \{ (0, 0), (1, 2) \}, \{ (1, 1), (1,
  2) \}, \{ (1, 1), (0, 0) \} \}$ is contained in $B$, $B$ cannot be strongly
  collapsed to a point. Hence $b 5, c 6 \in A$. As $A$ may not contain any
  loops in the contiguity class of $\{ \{ (0, 0), (2, 0) \}, \{ (2, 0), (1, 0)
  \}, \{ (1, 0), (2, 0) \} \}$, it must be that $c 1, c 4, d 1 \in B$.
  
  Similarly, as $B$ may not contain the loop $\{ \{ (0, 2), (0, 0) \}, \{ (0,
  0), (0, 1) \}, \{ (0, 1), (0, 2) \} \}$, or any loop that can be deformed to
  it while fixing $(0, 0)$, $c 5, d 5 \in A$, and by the previous argument $d
  3, d 4 \in B$, and then again $d 6 \in A$. This implies $d 2 \in B$, and
  then $e 5 \in A$.
  
  Thus, assuming that $\{ (0, 0), (2, 0) \} \nin B$ and $\{ (1, 0), (0, 0) \}
  \nin A$, out of necessity we have shown:
  
  \begin{figure}[H]\centering
  	\includegraphics[height=6cm,keepaspectratio]{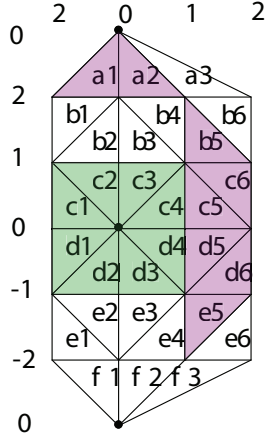}
    \caption{$A$ in purple and $B$ in green.}
  \end{figure}
  
  The edge $\{ (0, 0), (- 2, 1) \}$ cannot be contained in $A$, as then it
  would not be strongly collapsible containing a loop from the top $(0, 0)$ to
  the bottom one, hence $f 2, f 3 \in B$. But then $B$ cannot contain $\{ (-
  1, 0), (- 2, 0) \}$, hence $e 2, e 3 \in A$. But then there is a loop $\{ \{
  (- 1, 2), (- 1, 0) \}, \{ (- 1, 0), (- 1, 1) \}, \{ (- 1, 1), (- 1, 2) \} \}
  \subset A$, and $A$ cannot strongly collapse to a point.
  
  In a similar way, one can reach a contradiction by assuming $\{ (0, 0), (2,
  0) \} \nin B$ and $\{ (2, 0), (1, 0) \} \nin A$, or $\{ (2, 0), (1, 0) \}
  \nin B$ and $\{ (1, 0), (0, 0) \} \nin A$. Hence $K$ cannot be covered by
  $2$ strongly collapsible subcomplexes.
\end{proof}

\begin{acknowledgements}The author would like to thank E.Minuz
for bringing this problem to their attention and for fruitful discussions on
LS category.\end{acknowledgements}

\bibliographystyle{plain}
\bibliography{lscat}

\end{document}